\documentclass[a4paper,12pt,oneside,reqno]{amsart}
\usepackage{amsmath,amsfonts,amssymb,amsthm}
\renewcommand{\Bbb}{\mathbb}

\usepackage{bm}
\usepackage{caption}
\usepackage{cancel}
\usepackage{subcaption}
\DeclareCaptionSubType*[arabic]{figure}  
\usepackage{chngcntr}
\counterwithout{subfigure}{section}

\usepackage[new]{old-arrows}
\usepackage{graphicx}
\usepackage{color}
\usepackage[all]{xy}
\usepackage{url}
\usepackage{setspace}
\usepackage{booktabs}
\usepackage{longtable}
\usepackage{latexsym}
\usepackage[usenames,dvipsnames]{xcolor}
\usepackage{longtable}
\usepackage{framed}
\usepackage{tikz}
\usetikzlibrary{cd}
\usetikzlibrary{patterns}
\usetikzlibrary{shapes}
\usetikzlibrary{decorations.fractals}
\usepackage{colonequals}
\usepackage{dsfont}

\definecolor{c1}{RGB}{10,100,155}
\definecolor{c2}{RGB}{50,135,90}
\definecolor{c3}{RGB}{178,34,34}
\usepackage[pagebackref=true,colorlinks=true,linktocpage=true,linkbordercolor= c2	,citecolor=c1	, linkcolor=c2	,urlcolor=gray	]{hyperref}

\theoremstyle{plain}

\newtheorem{theorem}{Theorem}[section]
\newtheorem{lemma}[theorem]{Lemma}

\theoremstyle{definition}
\newtheorem{definition}[theorem]{Definition}
\newtheorem{example}[theorem]{Example}

\newtheoremstyle{noparens}
{}{}
{\itshape}{}
{\bfseries}{{\bf .}}
{ }
{\thmname{#1}\thmnumber{ #2}\mdseries\thmnote{ #3}}
\theoremstyle{noparens}
\newtheorem{theopar}[theorem]{Theorem}

\newtheorem{lempar}[theorem]{Lemma}

\newtheoremstyle{noparenstwo}
{}{}
{}{}
{\bfseries}{{\bf .}}
{ }
{\thmname{#1}\thmnumber{ #2}\mdseries\thmnote{ #3}}
\theoremstyle{noparenstwo}

\newtheorem{defpar}[theorem]{Definition}
\newtheorem{rempar}[theorem]{Remark}

\DeclareMathOperator{\cone}{\sigma}
\DeclareMathOperator{\vol}{vol} 
\DeclareMathOperator{\conv}{conv} 
\DeclareMathOperator{\modulo}{mod} 
\def\dedl#1\dedr{\left(\negmedspace \left(#1\right) \negmedspace \right)}
\DeclareMathOperator{\saw}{s}
\DeclareMathOperator{\str}{str}
\DeclareMathOperator{\nvol}{v}

\DeclareMathOperator{\sail}{sail_{\sigma}}
\DeclareMathOperator{\sailclosed}{\overline{sail}_{\sigma}}
\DeclareMathOperator{\sailhat}{sail_{\flexhat{\sigma}}}
\DeclareMathOperator{\sails}{sails}
\DeclareMathOperator{\fineFan}{\Delta_{\sails}}
\DeclareMathOperator{\LHS}{LHS} 
\DeclareMathOperator{\RHS}{RHS} 

\newcommand{\raygen}{u}
\newcommand{\raygenhat}{\flexhat{u}}
\newcommand{\mstar}{m_\sigma}
\newcommand{\ksig}{k_{\sigma}}
\newcommand{\mstark}[1]{m_{\sigma,#1}}

\newcommand{\mstarhat}{m_{\flexhat{\sigma}}}
\newcommand{\ksighat}{k_{\flexhat{\sigma}}}
\newcommand{\mstarhatk}[1]{m_{\flexhat{\sigma},#1}}

\newcommand{\atau}{a_\tau}
\newcommand{\nab}{\nabla_{\sigma}}
\newcommand{\nabint}{\nabla_{\sigma}^{\circ}}
\newcommand{\nabhat}{\nabla_{\flexhat{\sigma}}}

\newcommand{\convm}[1]{\conv[\mstar, #1]}
\newcommand{\convmhat}[1]{\conv[\mstarhat, #1]}
\newcommand{\N}{\mathbb N}
\newcommand{\Z}{\mathbb Z}
\newcommand{\Q}{\mathbb Q}
\newcommand{\R}{\mathbb R}
\newcommand{\C}{\mathbb C}
\newcommand{\T}{\mathbb T}

\makeatletter
  \newcommand\flausr{\@fleqntrue}
\makeatother

\newcommand\flexhat[1]{\widehat{#1}}

\begin{document}

\title[LDP polygons and the number 12 revisited]
{LDP polygons and the number 12 revisited}

 \author[U. Bücking]{Ulrike Bücking}
\address{Mathematisches Institut, Freie Universit\"at Berlin, 
Arnimallee 3, 14195 Berlin, Germany}
\email{buecking@math.fu-berlin.de}

 \author[C. Haase]{Christian Haase}
\email{haase@math.fu-berlin.de}

\author[K. Schaller]{Karin Schaller}
\email{kschaller@math.fu-berlin.de}

\author[J.-H. de Wiljes]{Jan-Hendrik de Wiljes}
\email{jan.dewiljes@math.fu-berlin.de}

\begin{abstract}
We give a combinatorial proof of a lattice point identity involving a lattice polygon and its dual, generalizing the formula $\operatorname{area}\left(\Delta \right) + \operatorname{area}\left(\Delta^* \right) = 6$ for reflexive $\Delta$.
The identity is equivalent to the stringy Libgober--Wood identity for toric log del Pezzo surfaces.
\end{abstract}

\maketitle
\setcounter{tocdepth}{1}

\thispagestyle{empty}

\section{Introduction}
\label{intro}

The goal of this article is to give a combinatorial proof of the following combinatorial identity: 
Consider the convex hull $\Delta \subseteq \R^2$ of $n_1, \ldots, n_k\in \Z^2$. 
Assume that each $n_i$ has coprime coordinates and that the origin is an interior point of $\Delta$. This data defines a piecewise linear function 
$\kappa_\Delta \colon
\R^2 \to \R$ via 
\begin{align*}
\kappa_\Delta(x) = - \min\left\{ \lambda \in \R_{\geq 0} \, \vert \, x \in
\lambda \Delta \right\}\,.
\end{align*}
The dual polygon is denoted by $\Delta^*$. Then
\begin{equation*} \label{originidentity}
6 \sum_{n \in \Delta \cap \Z^2} \left( \kappa_{\Delta}(n) + 1 
\right)^2 = \operatorname{area}\left(\Delta \right) +
\operatorname{area}\left(\Delta^* \right) \, .
\end{equation*}

This identity has been proven by Batyrev and the third
author~\cite[Corollary~4.5]{BS1} using a string-theoretic (allowing mild
singularities) variant of the Libgober--Wood identity for compact
complex manifolds~\cite[Proposition~2.3]{LW90}.
Conversely, we can obtain the stringy Libgober--Wood identity
for toric surfaces as a corollary of our combinatorial proof.
In our proof, we reduce this (global) statement to a local and cone-wise
statement, whose algebraic geometry
analogue could be of independent interest.

The formula as well as its variants in higher dimensions have a rich and colorful history. In the reflexive case where
both $\Delta$ and $\Delta^*$ have only integral vertices the formula
reduces to $\operatorname{area}\left(\Delta \right) + \operatorname{area}\left(\Delta^* \right) = 6$. Rodriguez-Villegas \& Poonen~\cite{PRV00} as well as Hille \& Skarke~\cite{HS02} prove non-convex and group theoretic generalizations of that latter formula, Kasprzyk \& Nill~\cite{KN12} relax the reflexivity hypothesis.
Still in dimension two, Haase \& Schicho~\cite{HS09} and also  Kołodziejczyk \& Olszewska~\cite{KO07} prove refinded inequalities, taking additional invariants into account.
In the open problem collection~\cite{CW05}, an equation for $3$-dimensional reflexive polytopes is stated, a combinatorial proof was sought.
Godinho, von Heymann \& Sabatini~\cite{GHS17} and Hofmann~\cite{Hofmann18} provide combinatorial proofs and higher-dimensional analogues, also for non-convex generalizations, but under a smoothness assumption for the underlying toric varieties. 
Finally, Batyrev and the third author \cite{BS1} remove this smoothness assumption, but need convexity.

\section{Notation and preliminaries}
\label{prelim}

This  chapter fixes our notation while recalling basic notions and provides a short summary of fundamental concepts used throughout the paper. Almost all of its content is well-known and the main references are \cite{CLS11, Ful93}.

\subsection{Polygons and Cones}

Let $N_{\R} \colonequals N \otimes_{\Z} \R \cong \R^2$ be a real vector space obtained by an extension of a $2$-dimensional lattice  $N \cong \Z^2$ (\emph{i.e.}, a free abelian group of rank $2$).
Furthermore, $M \colonequals \text{Hom}(N, \Z)$ denotes the dual lattice to $N$ 
and $\langle \cdot,\cdot \rangle \colon M \times N \to \Z$ the natural pairing
which extends to a pairing 
$\langle \cdot,\cdot \rangle \colon M_{\R} \times N_{\R} \to \R$, where 
$M_{\R}\colonequals M \otimes_{\Z} \R \cong \R^2$ is the corresponding 
real vector space to $M$.

Let $\Delta \subseteq N_{\R}$ be a $2$-dimensional 
 polytope (\emph{i.e.}, the convex hull $\conv(S)$
of a finite set $S \subseteq N_{\R}$),
 which will be called {\em polygon} in the following. \label{face} 
A {\em face} $\theta \preceq \Delta$ \emph{of}  $\Delta$ is an intersection of $\Delta$ with an affine hyperspace, \emph{i.e.}, there exists $m \in M_{\R}\setminus \left\{ 0\right\}$ and $ b \in \R$ such that 
$\theta = \Delta \cap H_{mb}$
with $H_{mb} \colonequals \left\{x \in N_{\R} \, \middle\vert \, \left<m,x \right> = b \right\} $. In particular, a {\em vertex} is a $0$-dimensional face and an {\em edge} a $1$-dimensional face.

If $\Delta \subseteq N_{\R}$ contains the origin $0 \in N$ in its interior, then one defines its {\em dual polygon} \label{dual} 
$\Delta^*$ as
\begin{align*}
\Delta^* \colonequals \left\{y \in M_{\R}\,  \middle\vert\,  \left<y,x \right> \geq -1  \; \forall x \in \Delta \right\} \subseteq M_{\R}\,.
\end{align*}
If $\theta \preceq \Delta$ is a face of $\Delta$, then
\begin{align*}
\theta^{*} \colonequals \{y \in \Delta^{*} \vert \left<y,x\right> = -1 \; \forall x \in \theta \} \preceq \Delta^{*} 
\end{align*}
is a face of $ \Delta^{*}$, the so-called {\em dual face of} $\theta$. 
Moreover, the duality between $\Delta$ and $\Delta^*$ implies a one-to-one order-reversing duality between $k$-dimensional faces \label{remdualface} $\theta \preceq \Delta$  of $\Delta$ and $(2-k-1)$-dimensional dual faces $\theta^* \preceq \Delta^*$  of $\Delta^*$ such that $\dim(\theta)+\dim(\theta^*)=1$.

A polygon $\Delta \subseteq N_{\R}$ is called {\em lattice polygon} if 
\begin{align*}\Delta =\conv(\Delta \cap N)\,,
\end{align*}
\emph{i.e.}, if all vertices of $\Delta$ belong to the
lattice $N$.

If a lattice polygon $\Delta \subseteq N_{\R}$  contains the origin $0 \in N$ in its interior and its dual polygon $\Delta^*$ is also a lattice polygon, then it is called {\em reflexive}.

\begin{definition} \label{defldp}
A lattice polygon $\Delta \subseteq N_{\R}$ containing
the origin $0 \in N$ in its interior such that all vertices of $\Delta$
are primitive lattice points in $N$ is called {\em LDP polygon},
where LDP is an abbreviation for \lq log del Pezzo\rq.
\end{definition}

In general, the vertices of the dual 
polygon $\Delta^*\subseteq M_\R$ to an LDP polygon $\Delta$ are not lattice points in $M$, \emph{i.e.},
$\Delta^*$ is in general a rational polygon.
If $\Delta\subseteq N_{\R}$ is a reflexive polygon, then the origin $0 \in N$ is the only interior lattice point of $\Delta$. Hence, all vertices of $\Delta$
are primitive lattice points in $N$, \emph{i.e.}, 
LDP polygons build a superclass of reflexive polygons.

Let $\Delta \subseteq N_{\R}$ be a lattice polygon.
Then we define $\nvol(\Delta)$ to be the {\em normalized volume of} $\Delta$,  
\emph{i.e.}, the positive integer 
\begin{align*}\nvol(\Delta) \colonequals 2!\cdot \vol_2 (\Delta)\,,
\end{align*}
where $\vol_2(\Delta)$ denotes the $2$-dimensional volume of $\Delta$ with respect to the lattice $N$. Note that $\vol_2(\Delta)=\operatorname{area}(\Delta)$ if $N= \Z^2$.
Similarly, we define the positive integer $\nvol(\theta) \colonequals k! \cdot \vol_k(\theta)$  for a $k$-dimensional face
$\theta \preceq \Delta$ of $\Delta$, where $\vol_k(\theta)$ denotes 
the $k$-dimensional volume of $\theta$ with respect to the sublattice $\langle \theta \rangle_{\R} \cap N$. 

If $\Delta$ has vertices in $N_{\Q} \colonequals N \otimes_{\Z} \Q$, 
\emph{i.e.}, $\Delta$ is a {\em rational
polygon}, then we can similarly define the positive rational number 
$\nvol(\theta)$ for any face
$\theta \preceq \Delta$. For this purpose, we consider an integer $l$ such that $l\Delta$ is
a lattice polygon and define for a $k$-dimensional face $\theta \preceq \Delta$ its normalized volume as
$\nvol(\theta) \colonequals \frac{1}{l^k} \nvol(l \theta)$.

Let $U \subseteq N_{\R}$ be a finite set. Then a (\emph{convex polyhedral}) \emph{cone} generated by $U$ is defined as the set
\begin{align*}
\cone\colonequals \text{cone}(U) = \{ {\textstyle \sum_{u \in U}} \lambda_u u \, \vert \,  \lambda_u \geq 0\}\,.
\end{align*}
If $U$ consists of $\R$-linear independent lattice vectors,
then the corresponding (half-open) {\em fundamental parallelogram} of  $U$ is
\begin{align*}
\Pi\colonequals \Pi(U)=  \{ {\textstyle \sum_{u \in U}} \lambda_u u  \, \vert \, 0\leq \lambda_u <1\}\,.
\end{align*}
Note that the normalized volume of the fundamental parallelogram equals the number of lattice points contained in it. 
Moreover, a $2$-dimensional cone $\cone$ is called {\em unimodular} if its ray generators  $u_1,u_2$ form a part of a $\Z$-basis of $N$. Note that in this case the fundamental parallelogram $\Pi(u_1,u_2)$ contains only one lattice point.

\begin{definition}
Let $\cone \subseteq N_{\R}$ be a $2$-dimensional cone. We define $\nab$ to be the convex hull of the origin and the primitive ray generators $u_1$ and $u_2$ of the given cone $\cone$, \emph{i.e.},
\begin{align*}
\nab \colonequals \conv(0,u_1,u_2)\,.
\end{align*}
We denote the {\em relative interior of} $\nab$ by $\nabint$. 

The {\em sail of} $\cone$ is the non-convex half-open lattice polygon defined as
\begin{align*}
\sail \colonequals \nab \setminus \conv(\nab \cap N \setminus \{0\})\,
\end{align*}
and its closure denoted by $\sailclosed$.
\end{definition}

The \emph{normalized volume} $\nvol(\cone)$ \emph{of} a $2$-dimensional cone $\cone$ is defined to be the normalized volume of the lattice polygon  $\theta_{\cone}$  obtained as the convex hull of the origin and all primitive ray generators of the given cone $\cone$, \emph{i.e.},
\begin{align*}
\nvol(\cone) \colonequals  \nvol(\theta_{\cone})\,.
\end{align*}

\subsection{Toric surfaces}

A {\em toric surface} $X$ is a normal variety of dimension $2$ over the field of complex numbers $\C$ containing a torus $\T \cong  (\C^*)^2$ as a Zariski open set such that the action of $(\C^*)^2$ 
on itself extends to an action on $X$. 

\begin{definition} \label{defspanfan}
Let $\Delta \subseteq N_{\R}$
be a lattice polygon with $0 \in \Delta^{\circ} \cap N$. 
We define 
$\Sigma_{\Delta}$ to be the {\em spanning fan}  of $\Delta$ in $N_{\R}$, \emph{i.e.},
$\Sigma_{\Delta} \colonequals \left\{\sigma_{\theta} \, \middle\vert \, \theta \preceq \Delta \right\}$, where
$\sigma_{\theta}$ is the cone $\R_{\geq 0} \theta$ spanned by the face $\theta \preceq \Delta$ of $\Delta$ with $\dim(\sigma_{\theta})= \dim(\theta)+1$. In particular, the spanning fan is a fan associated with 
an in general non-smooth normal projective toric surface $X_{\Sigma_{\Delta}}$. Moreover, one obtains
a resolution of singularities of $X_{\Sigma_{\Delta}}$ through the toric morphism $X_{\Sigma'_{\Delta}} \to X_{\Sigma_{\Delta}}$, where $\Sigma'_{\Delta}$ is a suitable refinement of $\Sigma_{\Delta}$. 
In our $2$-dimensional case, the rays of $\Sigma'_{\Delta}$ are spanned by all lattice points lying on the boundary of 
$\cup_{\sigma \in \Sigma_{\Delta}[2]} \sail$, where
 $\Sigma_{\Delta}[i]$ denotes the set of $i$-dimensional cones in the fan $\Sigma_{\Delta}$.
\end{definition}

We briefly recap that a normal projective surface  
is a \emph{log del Pezzo surface} if it has 
at worst log-terminal singularities 
and if its anticanonical divisor is an ample $\Q$-Cartier divisor. 
Moreover, toric log del Pezzo surfaces one-to-one correspond to LDP polygons. The fan $\Sigma$ defining a toric log del Pezzo surface $X$ is the spanning fan $\Sigma_{\Delta}$ of the corresponding LDP polygon $\Delta$. In particular,  any LDP polygon $\Delta$ is the convex hull of all primitive ray generators of elements 
in $\Sigma_{\Delta}[1]$. 

Let $\Delta \subseteq N_{\R}$ be a lattice polygon with $0 \in \Delta^{\circ} \cap N$.
Then there exists 
a $\Sigma_{\Delta}$-piecewise linear function 
$\kappa_{\Delta}: N_{\R} \to \R$
corresponding to the anticanonical divisor 
on $X_{\Sigma_{\Delta}}$ that is linear on each cone $\sigma$ of $\Sigma_{\Delta}$ and has 
value $-1$ on every 
primitive ray generator of  $1$-dimensional cones of $\Sigma_{\Delta}$. 

\medskip
In the rest of this subsection, we aim for introducing the 
stringy version of the Libgober--Wood identity from a 
geometric point of view, but restricted to log del Pezzo 
surfaces:
If $V$ is an arbitrary smooth projective surface,  
the {\em $E$-polynomial} of $V$ is defined as
\begin{align*}
E\left( V; u, v \right) \colonequals \sum_{0 \leq p,q \leq 2} 
(-1)^{p+q} h^{p,q}\left( V \right) u^{p} v^{q}\,,
\end{align*}
where $h^{p,q} \left( V \right)$ denote the Hodge numbers of $V$. 
The {\em stringy $E$-function} of a normal projective 
$\Q$-Gorenstein variety $X$ with at worst log-terminal singularities
is a rational algebraic function in two 
variables $u, v$ defined by the formula 
\begin{align*} 
E_{\str}(X; u, v) \colonequals \sum_{\emptyset \subseteq J \subseteq I}
E(D_J; u,v) \prod_{j \in J} \left( \frac{uv-1}{(uv)^{a_j +1} -1} -1 \right),
\end{align*}
where 
$\rho : Y \rightarrow X$ is some desingularization 
of $X$, whose exceptional locus 
is a union of  smooth irreducible divisors $D_1, \ldots, D_s$
with only simple normal crossings and
$K_Y=\rho^*  K_X  + \sum_{i=1}^s a_iD_i $
for some rational numbers $a_i >-1$. 
For any non-empty subset $J \subseteq I \colonequals \{1,\ldots,s\}$, we 
define  $D_J$ to be 
the smooth subvariety $\cap_{j \in J} D_j$. 
As a special case, this formula implies 
$E_{\str}(X;u,v) = E(X;u,v)$ if $X$ is smooth.

Let $X$ be a toric log del Pezzo surface associated with a fan $\Sigma$. Then 
the {\em stringy version of the Libgober--Wood identity} is given as
\begin{align*} 
\frac{d^2}{d u^2} E_{\str}\left(X;u,1\right)\Big\vert_{u=1} =
\frac16 c_2^{\str}(X) + \frac{1}{6} c_1(X).c_{1}^{\str}(X)
=\frac16 c_2^{\str}(X) + \frac{1}{6} c_1(X)^2\,,
\end{align*}
where $c^{\str}_k (X)$ denotes the
 {\em $k$-th stringy Chern class} introduced in  
\cite{MR2304329,MR2183846}. 
In particular, 
the $k$-th stringy Chern class of $X$ can be computed 
purely combinatorial via
\begin{align*}
c^{\str}_k (X)= \sum_{\sigma \in \Sigma(k)} \nvol(\sigma)\cdot [X_{\sigma} ]
\end{align*}
\cite{BS1}, where 
$[X_{\sigma}]$ denotes the class of the closed torus orbit $X_{\sigma}$ corresponding to a given cone $\sigma \in \Sigma$.
The general stringy version of the Libgober--Wood identity holding for any  projective variety  
with at worst log-terminal singularities can be found in \cite{Bat00}.

\section{Main theorem and its reduction to a local version}\label{secIntroProof}

We present a purely combinatorial proof of 
the following combinatorial identity that is equivalent to the stringy Libgober--Wood identity for log del Pezzo surfaces and relates LDP polygons to the number $12$:

\begin{theopar}[{\cite[Corollary 4.5]{BS1}}] \label{pezzo12} 
Let $\Delta \subseteq N_{\R}$ be an LDP polygon. 
Then 
\begin{align*}
12 \sum_{n \in \Delta \cap N} \left( \kappa_{\Delta}(n) + 1 \right)^2=\nvol\left(\Delta \right) +  \nvol \left(\Delta^* \right)\,,\end{align*}
where $\kappa_\Delta: N_{\R} \rightarrow \R,\, x \mapsto - \min\left\{ \lambda \in \R_{\geq 0} \, \vert \, x \in
\lambda \Delta \right\}$.
In particular, one always has $\nvol\left(\Delta \right) + 
 \nvol \left(\Delta^* \right) \geq 12$ 
and equality holds if and only if $\Delta$ is  reflexive. 
\end{theopar}

\begin{example}\label{ex1}
Let $\Delta  \subseteq N_{\R}$ be the LDP polygon given as the convex hull of $(0,-1)$, $(3,2)$, and $(-1,2)$ 
(cf.~Figure~\ref{ldpfig}\subref{ldpfig1}). Then Theorem~\ref{pezzo12} yields
\begin{align*}
 12 \sum_{n \in \Delta \cap N} \left( \kappa_{\Delta}(n) + 1 \right)^2
= 12 \cdot  ( 1^2+0.5^2+0.5^2) = 18
= 12 + 6
= \nvol \left(\Delta \right) +  \nvol \left(\Delta^* \right)\,,
\end{align*}
where the dual rational polygon $\Delta^*  \subseteq M_{\R}$ is the convex hull of $(0,-0.5)$, $(3,1)$, and $(-1,1)$ (cf.~Figure~\ref{ldpfig}\subref{ldpfig2}).
\end{example}

\begin{figure}
\centering
\begin{subfigure}[h]{0.4\textwidth}
\centering
\includegraphics[]{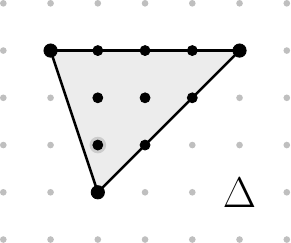}
\caption{}
\label{ldpfig1}
\end{subfigure}
\hspace{1cm}
\begin{subfigure}[h]{0.4\textwidth}
\centering
\includegraphics[]{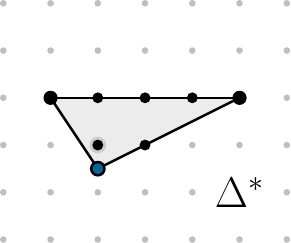}
\caption{}
\label{ldpfig2}
\end{subfigure}
\caption{{\bf LDP polygon $\boldsymbol{\Delta}$.} The origin is highlighted with a gray background. 
(A) Lattice polygon $\Delta$ with $ \nvol(\Delta)= 12$.
(B) Dual rational polygon $\Delta^*$ with rational vertex (blue) and $\nvol(\Delta^*)=6$.}
\label{ldpfig}
\end{figure} 

Our strategy relies on a decomposition of the identity in 
Theorem~\ref{pezzo12} using the spanning fan $\Sigma_{\Delta}$
of the given LDP polygon $\Delta$ and considering its $2$-dimensional cones separately. To do so, we need the following 

\begin{definition} \label{convmsigma}
Let $\cone \in \Sigma_{\Delta}[2]$  be a $2$-dimensional cone with primitive ray generators $\raygen_1$ and $\raygen_2 \in N$. Moreover,
let $\mstar \in M_\R$ be the vector dual to the edge $\raygen_1-\raygen_2$, \emph{i.e.}, $\langle \mstar , \raygen_1\rangle =-1= \langle \mstar ,\raygen_2\rangle$. We consider 
all $2$-dimensional cones $\sigma_1,\ldots, \sigma_{\ksig}$ in the refined
 fan $\Sigma_\Delta'$ of the given spanning fan $\Sigma_\Delta$ that are contained in $\sigma$. 
Therefore, we enumerate the corresponding $\ksig$ edges of the  $\sail$ from $\raygen_2$ to $\raygen_1$ consecutively, denote the corresponding dual vectors by $\mstark{1},\ldots, \mstark{\ksig}$, and
define
\begin{align*}
\convm{i} \colonequals \conv(\mstar,\mstark{i},\mstark{i+1})\,.
\end{align*}
\end{definition}

 Note that  the rays of all cones in  $\Sigma_\Delta'$ lying in $\sigma$ are spanned by the non-zero lattice points of $\sail$.
Moreover, $\mstark{1},\ldots, \mstark{\ksig}$ have integer coordinates as the resolved cones are unimodular while $\mstar$ may have rational coordinates.

\begin{theorem}\label{theoCones}
Let $\Delta \subseteq N_{\R}$ be an LDP polygon and $\cone \in \Sigma_\Delta[2]$ a
 $2$-dimensional cone of the spanning fan $\Sigma_\Delta$. Then
\begin{equation}\label{eq:sail}
12\sum\limits_{n\in \nabint \cap N}(\kappa(n)+1)^2
=\nvol\left(\nab \setminus
\sail \right)+ \sum\limits_{i=1}^{\ksig-1} \nvol(\convm{i})\,,
\end{equation}
where $\kappa \colonequals \kappa_\Delta\vert_{\cone}$ is a linear function  given as the 
restriction of the piecewise linear function $\kappa_\Delta$ to the cone $\cone$.
\end{theorem}

\begin{example}\label{ex2}
Let $\Delta \subseteq N_\R$ be the LDP polygon given in Example~\ref{ex1} and $\sigma \in \Sigma_{\Delta}[2]$  the $2$-dimensional cone of the spanning fan $\Sigma_{\Delta}$ with primitive ray generators $\raygen_1=(3,2)$ and 
$\raygen_2=(-1,2) \in N$ (cf. 
Figure~\ref{ldpfigII}\subref{ldpfigII1}). Then Theorem~\ref{theoCones} yields
\begin{align*}
12\sum\limits_{n\in \nabint \cap N}(\kappa(n)+1)^2
&= 12\sum\limits_{n\in \{(0,1),(1,1)\}}(\kappa(n)+1)^2
=6 = 5 + 0.5 + 0.5 \\
&=\nvol\left(\nab \setminus \sail \right)+ \nvol(\convm{1})+\nvol(\convm{2}) \\
&=\nvol\left(\nab \setminus \sail \right)+ \sum\limits_{i=1}^{\ksig-1} \nvol(\convm{i})\,,
\end{align*} 
where the $\nab=\conv((0,0),\raygen_1, \raygen_2)$, $\nab \setminus \sail= \conv(\raygen_1,\raygen_2, (0,1), (1,1))$, 
$\ksig=3$
 (cf. Figure~\ref{ldpfigII}\subref{ldpfigII1}, dotted edges), and 
 $\mstar=(0,-0.5)$, $\mstark{1}=(-1,-1)$, $\mstark{2}=(0,-1)$, $\mstark{3}=(1,-2)$  (cf. 
Figure~\ref{ldpfigII}\subref{ldpfigII2}).
\end{example}

\begin{figure}
\centering
\begin{subfigure}[h]{0.4\textwidth}
\centering
\includegraphics[]{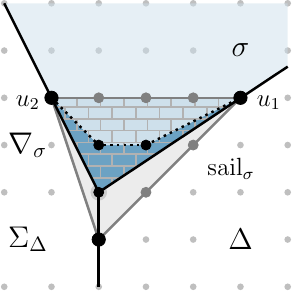}
\caption{}
\label{ldpfigII1}
\end{subfigure}
\hspace{1cm}
\begin{subfigure}[h]{0.4\textwidth}
\centering
\includegraphics[]{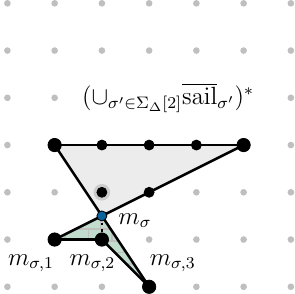}
\caption{}
\label{ldpfigII2}             
\end{subfigure}
\caption{{\bf Spanning fan $\boldsymbol{\Sigma_\Delta}$ with $\text{\bf sail}\boldsymbol{_\sigma}$ 
and dual polygon of $\boldsymbol{\cup_{\cone' \in \Sigma_{\Delta}[2]} \overline{\text{sail}}_{\cone'}}$.} 
(A) $2$-dimensional cone 
$\sigma$ (light and dark blue) of the spanning fan $\Sigma_\Delta$ of the LDP polygon $\Delta$ (gray), 
$\nabla_{\sigma}$ (brick) and $\sail$ (dark blue).
(B) Dual polygon of (non-convex) polygon $\cup_{\cone' \in \Sigma_{\Delta}[2]}  \overline{\text{sail}}_{\cone'}$ (gray) and $\cup_{i=1}^{\ksig-1} \convm{i}$ (green).}
\label{ldpfigII}
\end{figure} 

Theorem~\ref{theoCones}, which we prove combinatorially in Section~\ref{secProof}, is our main ingredient for our combinatorial proof of the identity in Theorem~\ref{pezzo12}. Summing up  Equation~\eqref{eq:sail} over all $2$-dimensional cones of the spanning fan $\Sigma_\Delta$ of our given LDP polygon $\Delta$, we obtain
\begin{align*}
 12\sum\limits_{n\in(\Delta\cap N)\setminus\{0\}}(\kappa(n)+1)^2& = \nvol(\Delta) - \nvol (\textstyle{\bigcup_{\sigma\in\Sigma_\Delta[2]}}
\sailclosed ) 	+\sum\limits_{\sigma\in\Sigma_\Delta[2]}\sum\limits_{i=1}^{\ksig-1} \nvol(\convm{i})\,.
\end{align*}

Comparing this identity with the one in Theorem~\ref{pezzo12}, it suffices to show 
\begin{align}\label{eq:diff1}
12&= \nvol \left(\Delta^* \right) +\nvol
(\textstyle{\bigcup_{\sigma\in\Sigma_\Delta[2]}}\sailclosed ) 	 
-\sum\limits_{\sigma\in\Sigma_\Delta[2]}\sum\limits_{i=1}^{\ksig-1} \nvol(\convm{i})\,.
\end{align}

We will consider the union $\cup_{\sigma\in\Sigma_\Delta[2]} \sailclosed$ 
of all closed sails as a non-convex polygon and denote it by $\fineFan$. 
Furthermore, we associate with it a fan $\Sigma_{\fineFan}$ having 
 rays that are spanned by the boundary lattice points of  $\fineFan$.
Note that this fan is unimodular, as all cones are unimodular cones by construction. In addition, this fan is \emph{complete}, meaning the union of its cones is the whole space $\R^2$. 
For such a fan, every $1$-dimensional cone $\tau$ with primitive ray generator $v$ is contained in precisely two $2$-dimensional cones $\sigma_l=\conv(v,v_l)$ and $\sigma_r=\conv(v,v_r)$ of this fan, where $v_l$ and $v_r$ are also primitive ray generators. Moreover, there exists a unique integer $\atau$ such that $v_l+v_r=\atau v$. Now we apply the following 
	
\begin{theopar}[{\cite[Subsection 8.1]{PRV00}}]\label{12thm}
 Let $\Sigma$ be a complete unimodular fan in $\R^2$. Then
\begin{align*}
12=\sum\limits_{\tau\in\Sigma[1]} (3-\atau)\,.
 \end{align*}
\end{theopar}

Combining this theorem with the $2$-dimensional property 
\begin{align*}
\nvol  (\textstyle{\bigcup_{\sigma\in\Sigma_\Delta[2]}}\sailclosed )  
=\sum_{\tau\in\Sigma_{\fineFan}[1]} 1
\end{align*}
and  Equation~\eqref{eq:diff1}, we arrive at
\begin{align*}
\sum\limits_{\tau\in\Sigma_{\fineFan}[1]} (2-\atau)
&=\nvol\left(\Delta^* \right) -\sum\limits_{\sigma\in\Sigma_{\Delta}[2]}\sum\limits_{i=1}^{\ksig-1}
\nvol(\convm{i})\,.
\end{align*}
In order to verify this identity, we again use the fact that $\Sigma_{\fineFan}$ is a complete unimodular fan in 
$\R^2$. The reasoning  
at the end of Section~8.1 in~\cite{PRV00}
can be applied to our case and states in particular that the sum
\begin{align*}
\sum\limits_{\tau\in\Sigma_{\fineFan}[1]} (2-\atau)
\end{align*}
equals the sum of signed lengths of dual edges $\tau^*$ corresponding to $\tau$.
Furthermore, the proof also shows that the sum of signed lengths of dual edges can be 
expressed as a sum of signed volumes. In particular,
\begin{align*}
\sum\limits_{\tau\in\Sigma_{\fineFan}[1]} (2-\atau)= \sum\limits_{\tau\in\Sigma_{\fineFan}[1]} 
\det(\tau^*) \,, 
\end{align*}
where $\det (\tau^*)$  is the determinant of the $2 \times 2$ matrix with the two vertices of $\tau^*$ as columns (respecting the direction of the edges in the chain) so that $\vert \det(\tau^*) \vert= \nvol(\conv(0,\tau^*))$. 
 It remains to deduce  
\begin{equation}\label{eq:sumdualvol}
\sum\limits_{\tau\in\Sigma_{\fineFan}[1]} \det(\tau^*) 
= \nvol \left(\Delta^* \right)- \sum\limits_{\sigma\in\Sigma_{\Delta}[2]}\sum\limits_{i=1}^{\ksig-1}
\nvol(\convm{i})\,.
\end{equation}

Let $\gamma$ be the closed curve corresponding to the chain of dual edges $\tau^*$ whose orientation is induced by the signs.
Observe that $\sum_{\tau\in\Sigma_{\fineFan}[1]} \det(\tau^*) = \int_\gamma \alpha$, where $\alpha$ is a 1-form such that $\frac{1}{2}d\alpha$ is the standard volume form on $\R^2$, 
see literature on differential forms, \emph{e.g.}, \cite[Section 37.3]{N21}.
We split the curve $\gamma$ into simple closed curves $\gamma_0$ and $\gamma_\sigma$ for $\sigma \in \Sigma_\Delta[2]$:
$\gamma_0$ runs through the boundary of $\Delta^*$ and $\gamma_\sigma$  through $\mstar, \mstark{1}, \ldots, \mstark{\ksig}, \mstar$. The integral splits into a sum of integrals over these simple closed curves, where 
$\int_{\gamma_0} \alpha = \nvol(\Delta^*)$ and $\int_{\gamma_\sigma} \alpha$ is the negative normalized volume of the area bounded by $\gamma_\sigma$ (the winding number is $-1$).
This area is subdivided into the triangles $\convm{i}$ (cf. Definition \ref{convmsigma}).

This shows Equation~\eqref{eq:sumdualvol} and thus finishes our combinatorial proof of the identity in Theorem~\ref{pezzo12}.
	
\begin{example}
We continue with the LDP polygon $\Delta$ studied in Example~\ref{ex1} and~\ref{ex2} and consider the dual polygon to $\cup_{\cone' \in \Sigma_{\Delta}[2]} \overline{\text{sail}}_{\cone'}$ (cf. Figure~\ref{figure34}). 
Equation \eqref{eq:sumdualvol} holds because
\begin{align*}
\sum\limits_{\tau\in\Sigma_{\fineFan}[1]} \det (\tau^*) = 4+2-1-1+1 
&= 6-0.5-0.5 \\
&=\nvol \left(\Delta^* \right)- \sum\limits_{\sigma\in\Sigma_{\Delta}[2]}\sum\limits_{i=1}^{\ksig-1}
\nvol(\convm{i})\,.
\end{align*}
\end{example}

\begin{figure}
\centering
\begin{subfigure}[h]{0.4\textwidth}
\centering 
\includegraphics[]{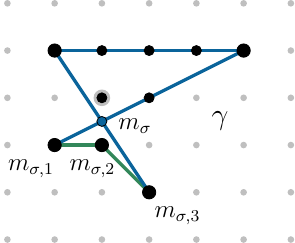}
\caption{}
\label{fig:dualPolArea}
\end{subfigure}
\hspace{1cm}
\begin{subfigure}[h]{0.4\textwidth}
\centering
\includegraphics[]{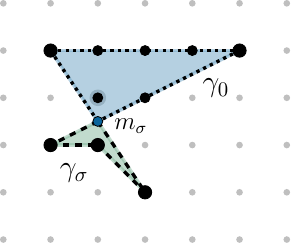}
\caption{}
\label{fig:dualPolAreaMarked2}          
\end{subfigure}
\caption{{\bf Dual polygon $\boldsymbol{(\cup_{\cone' \in \Sigma_{\Delta}[2]} \overline{\text{sail}}_{\cone'})^*}$
with closed curve $\boldsymbol{\gamma}$.} 
(A) Edges of $(\cup_{\cone' \in \Sigma_{\Delta}[2]} \overline{\text{sail}}_{\cone'})^*$  given as a closed chain $\gamma$ marked according to sign: blue for positive sign, green for negative sign. The interior intersection point of the two positive edges is $\mstar$.
(B) Signed volumes for the curve $\gamma$ split into $\gamma_0$ (blue area counted positive) and $\gamma_\sigma$ (green area counted negative).}
\label{figure34}
\end{figure} 

\section{Proving the cone-wise identity}\label{secProof}

We distinguish two cases in our proof of Theorem~\ref{theoCones}.
For unimodular cones, we easily see that the left hand side and the right hand side of Equation~\eqref{eq:sail} both vanish. For non-unimodular cones, a combinatorial proof by induction will be given in the rest of this section. 

Our reasoning is based on the fact that for any non-unimodular cone $\sigma$ with primitive ray generators $\raygen_1$ and $\raygen_2$ all interior lattice points of the fundamental parallelogram $\Pi(\raygen_1,\raygen_2)$ can be generated by some vector $w=\frac{1}{V}(a\raygen_1+\raygen_2)$, where $a\in\N$ and $V\colonequals \nvol(\sigma)= \nvol(\Pi(\raygen_1,\raygen_2))/2$. More precisely, these lattice points are represented in the $(\raygen_1,\raygen_2)$-basis as
\begin{equation}\label{eq:defiw}
    \lfloor iw\rfloor =\frac{1}{V}((ia \modulo V)\raygen_1+i\raygen_2)
\end{equation}
for $i=0,1,\dots, V$, see Figure~\ref{fig:coordinates}. 

\begin{figure}
\centering
\begin{subfigure}[h]{0.4\textwidth}
\centering
\includegraphics[]{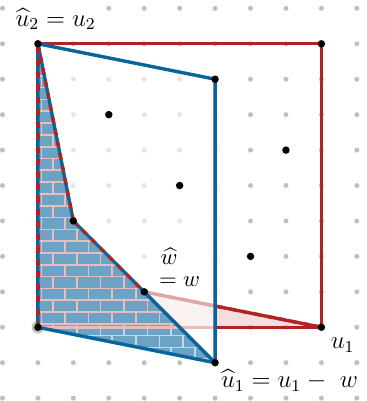}
\caption{}
\label{fig:coordinates1}
\end{subfigure}
\hspace{1cm}
\begin{subfigure}[h]{0.4\textwidth}
\centering
\includegraphics[]{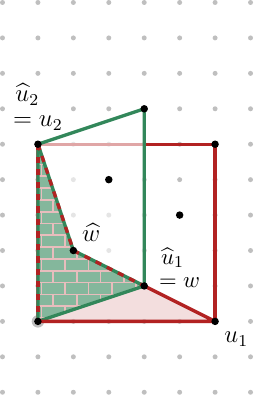}
\caption{}
\label{fig:coordinates2}
\end{subfigure}
\caption{{\bf Induction illustration of both cone generation operations.}
Both cases show the fundamental parallelograms
${\Pi(\raygen_1,\raygen_2)}$ of the given cone $\cone$ and ${\Pi(\raygenhat_1,\raygenhat_2)}$ of 
the cone $\flexhat{\cone}$ from which $\cone$ is constructed together with
the generating vectors ${w}$ and ${\flexhat{w}}$ for the respective base changes.
(A) Case I: $V=8, a=3$ and $\protect\flexhat{V}=5, \protect\flexhat{a}=3$.
(B) Case II: $V=5, a=3$ and $\protect\flexhat{V}=3, \protect\flexhat{a}=1$.}
\label{fig:coordinates}
\end{figure} 

Our main idea is to argue by induction over the volume $V= \nvol(\sigma)$ of a cone $\cone$ with 
primitive ray generators ${\raygen}_1$ and  ${\raygen}_2$, \emph{i.e.}, 
$\cone =  \text{cone}({\raygen}_1, {\raygen}_2)$.
For $V=1$ the cone is unimodular and thus Theorem~\ref{theoCones} holds.
We proceed by assuming that we are given a non-unimodular cone ${\cone}$ with 
volume $V >1$. 
Every such cone ${\cone}$ can be constructed from some other cone $\flexhat{\cone}$
with primitive ray generators $\flexhat{\raygen}_1$ and $\flexhat{\raygen}_2$ and strictly smaller 
volume, \emph{i.e.}, $\flexhat{V} \colonequals \nvol(\flexhat{\cone}) < \nvol(\sigma)=V$. In order to determine this cone $\flexhat{\cone}$, we consider the three consecutive lattice points $\raygen_1$, $w$, and $v$ on the boundary of $\sail$. 
As the two cones $ \text{cone}(\raygen_1,w)$ and $ \text{cone}(w,v)$ are unimodular, we deduce that $\raygen_1+v=\lambda w$ for some $\lambda\in\N$ with $\lambda\geq 2$. If $\lambda>2$, 
we define   $\flexhat{\cone}$ to be the cone generated by $\flexhat{\raygen}_1 \colonequals \raygen_1-w$ and 
$\flexhat{\raygen}_2 \colonequals \raygen_2$. Note that  $\flexhat{\cone}$
 has the same lattice points in its sail as $\cone$ except $\raygen_1$ (which is replaced by $\raygen_1-w$) and has strictly smaller volume, see Figure~\ref{fig:coordinates}\ref{fig:coordinates1} for an illustration. 
 If $\lambda=2$, the three lattice points $\raygen_1$, $w$, and $v$ are collinear and
 we define $\flexhat{\cone}$ to be the cone generated by 
 $\flexhat{\raygen}_1 \colonequals w$ and $\flexhat{\raygen}_2 \colonequals \raygen_2$. 
 Note that  $\overline{\text{sail}}_{\flexhat{\cone}} \cap N = \sailclosed \cap \, N$ and $\flexhat{\cone}$ has strictly smaller volume than 
 $\cone$,
see Figure~\ref{fig:coordinates}\ref{fig:coordinates2}.
Therefore, every non-unimodular cone ${\cone}$ can be obtained from some other cone $\flexhat{\cone}= \text{cone}(\flexhat{\raygen}_1,\flexhat{\raygen}_2)$ with strictly smaller volume by one of the following two operations: 
\begingroup
\flausr
\begin{alignat*}{4}
\qquad&\phantom{\mathrm{I}}\mathrm{I.}\ \flexhat{\cone}\rightarrow{\cone}= \text{cone}(\flexhat{\raygen}_1+\flexhat{w},\flexhat{\raygen}_2),&&\quad \flexhat{\raygen}_1\mapsto{\raygen}_1=\flexhat{\raygen}_1+\flexhat{w},&&\quad \flexhat{\raygen}_2\mapsto{\raygen}_2=\flexhat{\raygen}_2\, , &&
\\
\qquad&\mathrm{II.}\ 
\flexhat{\cone}\rightarrow{\cone}= \text{cone}(2\flexhat{\raygen}_1-\flexhat{w},\flexhat{\raygen}_2),&&\quad \flexhat{\raygen}_1\mapsto{\raygen}_1=2\flexhat{\raygen}_1-\flexhat{w},&&\quad \flexhat{\raygen}_2\mapsto{\raygen}_2=\flexhat{\raygen}_2\,, &&
\end{alignat*}
\endgroup
where $\flexhat{w}$ is defined analogously as $w$ above, see Figure~\ref{fig:coordinates}.
By construction, we can easily deduce the properties
\begingroup
\flausr
\begin{alignat}{4}
\qquad&\phantom{\mathrm{I}}\mathrm{I.}\  \flexhat{V}\mapsto{V}=\flexhat{V}+\flexhat{a},&&\quad a=\flexhat{a},&&\quad {w}=\flexhat{w}\, , &&\label{eq:volChange_CaseI}\\
\qquad&\mathrm{II.}\  \flexhat{V}\mapsto{V}=2\flexhat{V}-\flexhat{a},&&\quad a=\flexhat{V},&&\quad {w}=\flexhat{\raygen}_1\,. &&\label{eq:volChange_CaseII}
\end{alignat}
\endgroup

As induction hypothesis we assume that Equation~\eqref{eq:sail} holds for all cones $\flexhat{\cone}$ with volume $\flexhat{V}$ strictly smaller than ${V}$. Therefore, it suffices to show that Equation~\eqref{eq:sail} still holds when the cone $\flexhat{\cone}$ is changed to ${\cone}$. We consider the left hand side (LHS) and the right hand side (RHS) of Equation~\eqref{eq:sail} separately and determine the differences of the new and old values associated to ${\cone}$ and $\flexhat{\cone}$, that is ${\LHS}-\flexhat{\LHS}$ and ${\RHS}-\flexhat{\RHS}$, respectively. Comparing these differences, we will see that they coincide. This finishes the induction step.

\subsection{Left hand side of Equation~(\ref{eq:sail})}

First, we will express  the left hand side of Equation~\eqref{eq:sail} 
 in terms of the volume and the sawtooth function.

\begin{defpar}[{\cite[Chapter 1, Introduction]{RG72}}]
\label{defnsawtooth}
\label{defndedekindsum}
Let $x$ be a rational number. Then 
\begin{align*}
\dedl x\dedr \colonequals \left\{ 
\begin{array}{ll}
x -[x] - 1/2 & \text{if } x \in \Q \setminus \Z, \\
0 & \text{if } x \in \Z \\
\end{array} \right.
\end{align*}
defines the \emph{sawtooth function} of period $1$, where 
$[x]$ denotes the greatest integer not exceeding $x$.
Given integers $h,k$ with $\gcd(h,k)=1$ and $k \geq 1$, the \emph{Dedekind sum} is defined as
\begin{align*}
\saw(h,k)&\colonequals \sum_{i =1}^k \dedl\frac{h i}{k}\dedr \dedl\frac{i}{k}\dedr\,.
\end{align*}
\end{defpar}

\begin{rempar}\label{rem1}
Let $h,k,m,$ and $i$ be integers. By the periodicity of the sawtooth function we immediately see that 
$\dedl \frac{(h-mk)i}{k}\dedr =\dedl\frac{h i}{k}\dedr$. Thus
\begin{align*}
\saw(h,k) = \saw(h-mk,k) \,\, \text{ and } \,\,\saw(-h,k) = - \saw(h,k)\,,
\end{align*}
where the last equation holds since the sawtooth function is odd, that is $\dedl -x \dedr=- \dedl x\dedr$ \cite[Chapter 3, Elementary Properties]{RG72}.
\end{rempar}

\begin{lempar}[{\cite[Chapter 2, Lemma 2, Theorem 1]{RG72}}]
\label{lem12}
Let $h$ and $k$ be two integers with $\gcd(h,k)=1$. Then  
\begin{align*}
\saw(1,k)= - \frac{1}{4} +\frac{1}{6k}+\frac{k}{12}
=\frac{1}{12k} (k-1)(k-2)
\end{align*}
and
\begin{align*}
\saw(h,k) + \saw(k,h) = - \frac{1}{4}  + \frac{1}{12} \left(\frac{h}{k}+\frac{1}{hk}+\frac{k}{h} \right)\,.
\end{align*}
\end{lempar}

\begin{lemma}\label{lemKappa}
	For every $2$-dimensional cone $\cone$ with primitive ray generators $\raygen_1$ and $\raygen_2$, 
	we have
	\begin{align}\label{eq:LHS}
	12\sum\limits_{n\in\nabint\cap N}(\kappa(n)+1)^2&=\frac{(V-1)(V-2)}{V} +12 \cdot \saw(a,V)\,,
	\end{align}
	where $V= \nvol(\sigma)$ and $a\in\N$ is such that all interior lattice points of $\Pi(\raygen_1,\raygen_2)$ are generated by $w=\frac{1}{V}(a\raygen_1+\raygen_2)$.
\end{lemma}

\begin{proof}
Without loss of generality, we restrict ourselves to non-unimodular cones $\cone=\text{cone}(\raygen_1,\raygen_2)$.
Since $\raygen_1$ and $\raygen_2$ are primitive and $Vw=a\raygen_1+\raygen_2$, we have $\gcd(V,a)=1$. Furthermore, we denote by $(\raygen_1^*,\raygen_2^*)$ the dual basis to $(\raygen_1,\raygen_2)$ with respect to the standard scalar product. Observe that $\kappa=-\raygen_1^*-\raygen_2^*$ by construction. Therefore, we get
\begin{align*}
2\sum\limits_{n\in\nabint\cap N}(\kappa(n)+1)^2=\sum\limits_{i=1}^{V-1}(\kappa(\lfloor iw\rfloor)+1)^2
\end{align*}
which  follows from the symmetry of $\kappa$. Furthermore, as $\kappa=-\raygen_1^*-\raygen_2^*$ and $\gcd(V,a)=1$, we deduce from Equation~\eqref{eq:defiw} and Definition~\ref{defnsawtooth} that 
\begin{align*}
2V \left(1+\kappa(\lfloor iw\rfloor) +\dedl \frac{ai}{V}\dedr\right)&= 
2V- 2(ia \modulo V) -2i + 2ai - 2V\left[ \frac{ai}{V} \right] -V \\
&= V-2i+ 2 \left(ai- (ia \modulo V) -V\left[ \frac{ai}{V} \right]\right) \\
&= 2V-i
\end{align*}
holds for  $i=1,\dots, V-1$. This yields
\begin{align*}
\kappa(\lfloor iw\rfloor)+1=-\dedl \frac{ai}{V}\dedr+\frac{1}{2}-\frac{i}{V}\,.
\end{align*}
As $\gcd(V,a)=1$, the set of values of $\dedl \frac{ai}{V}\dedr^2$ for $i=1,\dots, V-1$ agrees with the set of values of 
$\dedl \frac{j}{V}\dedr^2$ for $j=1,\dots, V-1$. Therefore,
\begin{align*}
\sum\limits_{i=1}^{V-1}\dedl \frac{ai}{V}\dedr ^2
&= \sum\limits_{j=1}^{V-1} \dedl\frac{j}{V}\dedr^2 = \saw(1,V) 
= \frac{1}{12V}(V-1)(V-2)\,,
\end{align*}
where we used Definition~\ref{defnsawtooth}, Lemma~\ref{lem12}, and
the fact $\dedl\frac{V}{V}\dedr=0$ for the second equality.
It is well known that $\sum_{i=1}^{V-1} \dedl\frac{ai}{V} \dedr=0$ (as $ \dedl\frac{ai}{V} \dedr+ \dedl\frac{a(V-i)}{V} \dedr=0$ for $i=1,\dots, V-1$ if $\gcd(V,a)=1$, analogously to Remark~\ref{rem1}). Thus,
\begin{align*}
\sum\limits_{i=1}^{V-1}\frac{i}{V} \dedl\frac{ai}{V} \dedr
&=\sum\limits_{i=1}^{V-1}\left(\frac{i}{V}-\left[\frac{i}{V}\right]-\frac{1}{2}\right) \dedl\frac{ai}{V} \dedr
=\sum\limits_{i=1}^{V-1} \dedl\frac{i}{V} \dedr \dedl\frac{ai}{V} \dedr=	\saw(a,V)\,.
\end{align*}
Combining everything, we obtain
\begin{align*}
\sum\limits_{i=1}^{V-1}(\kappa(\lfloor iw\rfloor)+1)^2
&=\sum\limits_{i=1}^{V-1}\left(- \dedl\frac{ai}{V}\dedr +\frac{1}{2}-\frac{i}{V}\right)^2\\
&=\sum\limits_{i=1}^{V-1} \dedl\frac{ai}{V}\dedr^2 -\sum\limits_{i=1}^{V-1}\dedl\frac{ai}{V}\dedr+2\sum\limits_{i=1}^{V-1}\frac{i}{V} \dedl\frac{ai}{V}\dedr \\
&\phantom{===}-\frac{1}{V}\sum\limits_{i=1}^{V-1}i+\frac{1}{4}(V-1)+\frac{1}{V^2}\sum\limits_{i=1}^{V-1}i^2\\
&=\frac{1}{12V}(V-1)(V-2)+2 \saw(a,V)+\frac{1}{12V}(V-1)(V-2)\\
&=\frac{1}{6V}(V-1)(V-2)+2 \saw(a,V)\,.
\end{align*}
\end{proof}

In the following, we will consider the difference ${\LHS}-\flexhat{\LHS}$ for Case I and II separately:

\paragraph{\bf Case I}
Using  Lemma~\ref{lemKappa} and Equation~\eqref{eq:volChange_CaseI}, we get
\begin{align*}
{\LHS}-\flexhat{\LHS}&=\frac{1}{\flexhat{V}+\flexhat{a}}(\flexhat{V}+\flexhat{a}-1)(\flexhat{V}+\flexhat{a}-2)+12 \saw(\flexhat{a},\flexhat{V}+\flexhat{a})\\
&\quad -\frac{1}{\flexhat{V}}(\flexhat{V}-1)(\flexhat{V}-2)-12 \saw(\flexhat{a},\flexhat{V})\,.
\end{align*}
The reciprocity law (Lemma~\ref{lem12}) and the periodicity for Dedekind sums (Remark~\ref{rem1})
imply

\begin{align*}
12\saw(\flexhat{a},\flexhat{V}+\flexhat{a}) 
&=12\left(-\saw(\flexhat{V}+\flexhat{a},\flexhat{a})-\frac{1}{4}+\frac{1}{12}\left(\frac{\flexhat{a}}{\flexhat{V}+\flexhat{a}}+\frac{1}{\flexhat{a}(\flexhat{V}+\flexhat{a})}+\frac{\flexhat{V}+\flexhat{a}}{\flexhat{a}}\right)\right)\\ 
&= -12\saw(\flexhat{V},\flexhat{a})-3
+\frac{\flexhat{a}}{\flexhat{V}+\flexhat{a}}+\frac{1}{\flexhat{a}(\flexhat{V}+\flexhat{a})}+\frac{\flexhat{V}+\flexhat{a}}{\flexhat{a}}\\ 
&= 12\left(\saw(\flexhat{a},\flexhat{V})+\frac{1}{4}-\frac{1}{12}\left(\frac{\flexhat{a}}{\flexhat{V}}+\frac{1}{\flexhat{a}\flexhat{V}}+\frac{\flexhat{V}}{\flexhat{a}}\right)\right)\\ 
&\qquad -3
+\frac{\flexhat{a}}{\flexhat{V}+\flexhat{a}}+\frac{1}{\flexhat{a}(\flexhat{V}+\flexhat{a})}+\frac{\flexhat{V}+\flexhat{a}}{\flexhat{a}})\\
&=12\saw(\flexhat{a},\flexhat{V})+\frac{-\flexhat{a}^2+\flexhat{a}\flexhat{V}+\flexhat{V}^2-1}{\flexhat{V}(\flexhat{V}+\flexhat{a})}\,.
\end{align*}
Simplifying
\begin{align*}
\frac{1}{\flexhat{V}+\flexhat{a}}(\flexhat{V}+\flexhat{a}-1)(\flexhat{V}+\flexhat{a}-2)=\frac{1}{\flexhat{V}}(\flexhat{V}-1)(\flexhat{V}-2)+\frac{\flexhat{a}(\flexhat{a}\flexhat{V}+\flexhat{V}^2-2)}{\flexhat{V}(\flexhat{V}+\flexhat{a})}\,,
\end{align*}
we arrive at
\begin{align}\label{eq:diffLHS2}
{\LHS}-\flexhat{\LHS}
&=\frac{(\flexhat{a}+1)(\flexhat{V}-1)(\flexhat{V}+\flexhat{a}+1)}{\flexhat{V}(\flexhat{V}+\flexhat{a})}=(\flexhat{a}+1)\left(1-\frac{\flexhat{a}+1}{\flexhat{V}(\flexhat{V}+\flexhat{a})}\right)\,.
\end{align}
\bigskip

\paragraph{\bf Case II}
Using Lemma~\ref{lemKappa} and Equation~\eqref{eq:volChange_CaseII}, we similarly obtain
\begin{align*}
{\LHS}-\flexhat{\LHS}&=\frac{1}{2\flexhat{V}-\flexhat{a}}(2\flexhat{V}-\flexhat{a}-1)(2\flexhat{V}-\flexhat{a}-2)+12\saw(\flexhat{V},2\flexhat{V}-\flexhat{a}) \\
&\quad -\frac{1}{\flexhat{V}}(\flexhat{V}-1)(\flexhat{V}-2)-12\saw(\flexhat{a},\flexhat{V})\,.
\end{align*}
Again, the reciprocity law (Lemma~\ref{lem12}) and elementary properties of Dedekind sums (Remark~\ref{rem1}) imply
\begin{align*}
12\saw(\flexhat{V},2\flexhat{V}-\flexhat{a})
&=-12 \saw(2\flexhat{V}-\flexhat{a},\flexhat{V})-3 +
\frac{\flexhat{a}}{2\flexhat{V}-\flexhat{a}}+\frac{1}{\flexhat{a}(2\flexhat{V}-\flexhat{a})}+\frac{2\flexhat{V}-\flexhat{a}}{\flexhat{a}}\\
&=12\,\saw(\flexhat{a},\flexhat{V})-3+\frac{\flexhat{V}^2+(2\flexhat{V}-\flexhat{a})^2+1}{\flexhat{V}(2\flexhat{V}-\flexhat{a})}\,.
\end{align*}
As
\begin{align*}
\frac{1}{2\flexhat{V}-\flexhat{a}}(2\flexhat{V}-\flexhat{a}-1)(2\flexhat{V}-\flexhat{a}-2) =\frac{1}{\flexhat{V}}(\flexhat{V}-1)(\flexhat{V}-2)+\frac{(\flexhat{V}-\flexhat{a})(-\flexhat{a}\flexhat{V}+2\flexhat{V}^2-2)}{\flexhat{V}(2\flexhat{V}-\flexhat{a})}\,,
\end{align*}
we finally obtain
\begin{align}\label{eq:diffLHS1}
{\LHS}-\flexhat{\LHS}
&=\frac{(\flexhat{V}+1)(\flexhat{V}-\flexhat{a}-1)(2\flexhat{V}-\flexhat{a}-1)}{\flexhat{V}(2\flexhat{V}-\flexhat{a})} \notag\\
&=(\flexhat{V}-\flexhat{a}-1)\left(1+\frac{\flexhat{V}-\flexhat{a}-1}{\flexhat{V}(2\flexhat{V}-\flexhat{a})}\right)\,.
\end{align}

\begin{figure}
\centering
\begin{subfigure}[h]{0.4\textwidth}
\centering
\includegraphics[]{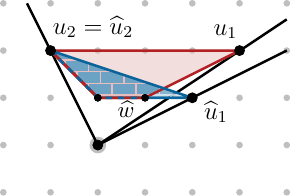}
\caption{}
\label{fig:rhscaseisailnew5a}
\end{subfigure}
\hspace{1cm}
\begin{subfigure}[h]{0.4\textwidth}
\centering
\includegraphics[]{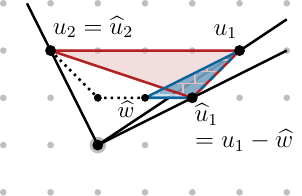}
\caption{}
\label{fig:rhscaseisailnew5b}
\end{subfigure}
\caption{{\bf Case I -- Illustration of considered domains in decomposition of 
$\boldsymbol{\nab \setminus} \text{sail}_{\boldsymbol{\sigma}}$.}
Here: $V=8, a=3$ and $\protect\flexhat{V}=5, \protect\flexhat{a}=3$. 
(A) $\nabhat \setminus \sailhat$ (blue) and $\nab \setminus \sail$ (red).
(B)  $\conv(\flexhat{\raygen}_1,\flexhat{\raygen}_1 + \flexhat{w},\flexhat{w})$ (blue) and
$\conv(\flexhat{\raygen}_1,
\flexhat{\raygen}_1+\flexhat{w},\flexhat{\raygen}_2)$ (red).}
\label{fig:rhscaseisailnew}
\end{figure} 

\subsection{Right hand side of Equation~(\ref{eq:sail})}\label{subsec_RHS}

The two summands 
\begin{align*}
\nvol\left(\nab \setminus \sail \right) \,\, \text{ and } \,\,
\sum\limits_{i=1}^{\ksig-1} \nvol(\convm{i})
\end{align*}
on the right hand side of Equation~\eqref{eq:sail} will be considered separately for each case. 

\paragraph{\bf Case I}
By our construction, we have
\begin{align*}
\nab \setminus\sail= \overline{\left(\left(\nabhat\setminus\sailhat\right)\cup 
\conv(\flexhat{\raygen}_1,\flexhat{\raygen}_1+\flexhat{w},\flexhat{\raygen}_2) \right) \setminus \conv(\flexhat{\raygen}_1,\flexhat{\raygen}_1+\flexhat{w},\flexhat{w})}
\end{align*}
as illustrated in Figure~\ref{fig:rhscaseisailnew}.
Since
\begin{align*}
\conv(\flexhat{\raygen}_1,\flexhat{\raygen}_1+\flexhat{w},\flexhat{w})\subseteq\left(\nabhat\setminus\sailhat\right)\cup\conv(\flexhat{\raygen}_1,\flexhat{\raygen}_1+\flexhat{w},\flexhat{\raygen}_2)
\end{align*}
and
\begin{align*}
\nvol\left((\nabhat\setminus\sailhat)\cap\conv(\flexhat{\raygen}_1,\flexhat{\raygen}_1+\flexhat{w},\flexhat{\raygen}_2)\right)=0\,,
\end{align*}
we obtain
\begin{align*}
&\nvol\left(\nab\setminus\sail\right) -\nvol\left(\nabhat\setminus\sailhat\right) \\
&\qquad\qquad =\nvol\left(\conv(\raygenhat_1,\raygenhat_1+\flexhat{w},\raygenhat_2)\right)-\nvol\left(\conv(\raygenhat_1,\raygenhat_1+\flexhat{w},\flexhat{w})\right)\\ 
&\qquad\qquad =\det(\flexhat{w},\raygenhat_2-\raygenhat_1)-\det(\flexhat{w},\flexhat{w}-\raygenhat_1)
=\det(\flexhat{w},\raygenhat_2-\flexhat{w})
=\det(\flexhat{w},\raygenhat_2)\\ \displaybreak
&\qquad\qquad =\frac{\flexhat{a}}{\flexhat{V}}\cdot \flexhat{V}
=\flexhat{a}
\end{align*}
by utilizing $\det(\flexhat{\raygen}_1,\flexhat{\raygen}_2)=\flexhat{V}$.

For the second summand of the right hand side of Equation~\eqref{eq:sail}, we need to determine how the involved functionals   behave when we change from $\flexhat{\cone}$ to ${\cone}$.
Recall $\mstarhat=-\raygenhat_1^*-\raygenhat_2^*$. Since ${\raygen}_1=\raygenhat_1+\ \flexhat{w}$ and ${\raygen}_2=\raygenhat_2$, we have 
\begin{align*}
\mstar=-\frac{\flexhat{V}-1}{\flexhat{V}+\flexhat{a}}\raygenhat_1^*-\raygenhat_2^*=\mstarhat+\frac{\flexhat{a}+1}{\flexhat{V}+\flexhat{a}}\raygenhat_1^*\,.
\end{align*}

\begin{figure}
\centering
\includegraphics[]{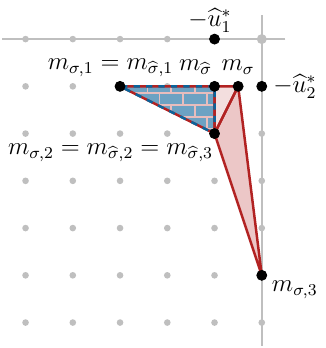}
\caption{{\bf Case I -- dual perspective} for illustration in Figure~\ref{fig:rhscaseisailnew} with 
$\cup_{i=1}^{\ksighat-1} \convmhat{i}$ (blue) and  
$ \cup_{i=1}^{\ksig-1}\convm{i}$ (red).}
\label{fig:rhscaseidual}
\end{figure}

As explained in Section~\ref{secIntroProof}, the functionals $\mstarhatk{1},\mstarhatk{2},\dots,\mstarhatk{\ksig}$ may be associated with edges of $\overline{\text{sail}}_{\flexhat{\cone}}$ not incident to $0$ (and similarly for $\sailclosed$). We enumerate these edges of the sail starting from the edge incident to $\raygenhat_2$ 
(and finishing with the edge incident to $\raygenhat_1$).
By construction, $\sailclosed$ has the same edges as $\overline{\text{sail}}_{\flexhat{\cone}}$, except for the last edge connecting $\flexhat{w}$ to $\raygenhat_1$ and $\flexhat{w}$ to $\raygenhat_1+\ \flexhat{w}$, respectively 
(cf. Figure~\ref{fig:rhscaseisailnew}\subref{fig:rhscaseisailnew5a}). Accordingly, this is also true for the functionals, except that $\mstarhatk{\ksighat}$ is replaced by $\mstark{\ksig}=-\flexhat{V}\raygenhat_2^*$ (cf. Figure~\ref{fig:rhscaseidual}).
As 
\begin{align*}
\langle \mstarhatk{1},\raygenhat_2\rangle=-1 \,\, \text{ and } \,\,
\langle \mstarhatk{1},\dfrac{1}{\flexhat{V}}\raygenhat_1+ \dfrac{\flexhat{b}}{\flexhat{V}}\raygenhat_2\rangle=-1\,,
\end{align*}
we have
\begin{align}\label{eq:m1}
\mstarhatk{1}=(\flexhat{b}-\flexhat{V})\raygenhat_1^*-\raygenhat_2^*\,,
\end{align}
where $\flexhat{b}\in[1,V]$ is the multiplicative inverse of $\flexhat{a}$ modulo $\flexhat{V}$, \emph{i.e.}, $\flexhat{a}\cdot \flexhat{b}=1$ mod $\flexhat{V}$.
Furthermore,
\begin{align}\label{eq:mk}
\mstarhatk{\ksighat}=-\raygenhat_1^*+(\flexhat{a}-\flexhat{V})\raygenhat_2^*
\end{align}
because  $\langle\mstarhatk{\ksighat},\raygenhat_1\rangle=-1$ and $\langle \mstarhatk{\ksighat},\flexhat{w}\rangle=\langle \mstarhatk{\ksighat}, \dfrac{\flexhat{a}}{\flexhat{V}}\raygenhat_1+\dfrac{1}{\flexhat{V}}\raygenhat_2\rangle=-1$.
Additionally, $\mstarhat$, $\mstar$ and $\mstarhatk{1}$ are collinear (as they all take value $-1$ on $\raygenhat_2$), see Figure~\ref{fig:rhscaseidual}. 
Therefore, with $\ksighat=\ksig$, we have
\begin{align*}
\bigcup\limits_{i=1}^{\ksighat-1} \convmhat{i}\subseteq 
 \bigcup\limits_{i=1}^{\ksig-1}\convm{i}\,.
 \end{align*}
Hence, we get (cf. Figure~\ref{fig:rhscaseidual})
\begin{align*}
&\sum\limits_{i=1}^{\ksig-1} \nvol(\convm{i})-
\sum\limits_{i=1}^{\ksighat-1} \nvol(\convmhat{i})\\ 
&= \nvol({\conv(\mstar,\mstarhat,\mstarhatk{\ksighat})}) + \nvol({\conv(\mstar,\mstarhatk{\ksighat},\mstark{
\ksig})})\\ 
&=\det(\mstarhatk{\ksighat}-\mstarhat,\mstar-\mstarhat) +\det(\mstark{\ksig}-\mstarhatk{\ksighat},\mstar-\mstarhatk{\ksighat})\\ 
&=\det\left((\flexhat{a}+1-\flexhat{V})\raygenhat_2^*, \frac{\flexhat{a}+1}{\flexhat{V}+\flexhat{a}}\raygenhat_1^*\right) +\det\left(\raygenhat_1^*-\flexhat{a}\raygenhat_2^*, (\flexhat{V}-\flexhat{a}-1)\raygenhat_2^*+\frac{\flexhat{a}+1}{\flexhat{V}+\flexhat{a}}\raygenhat_1^*\right)\\
&=-\frac{(\flexhat{V}-\flexhat{a}-1)(\flexhat{a}+1)}{\flexhat{V}+\flexhat{a}}\det(\raygenhat_2^*,\raygenhat_1^*) -\frac{\flexhat{a}(\flexhat{a}+1)}{\flexhat{V}+\flexhat{a}}\det(\raygenhat_2^*,\raygenhat_1^*) +(\flexhat{V}-\flexhat{a}-1)\det(\raygenhat_1^*,\raygenhat_2^*)\\
&=\frac{(\flexhat{V}-\flexhat{a}-1)(\flexhat{a}+1)}{\flexhat{V}(\flexhat{V}+\flexhat{a})}+\frac{\flexhat{a}(\flexhat{a}+1)}{\flexhat{V}(\flexhat{V}+\flexhat{a})} +\frac{\flexhat{V}-\flexhat{a}-1}{\flexhat{V}}
=1-\frac{(\flexhat{a}+1)^2}{\flexhat{V}(\flexhat{V}+\flexhat{a})}\,.
\end{align*}
Combining everything yields
\begin{equation*} \label{rhs-rhscase1}
{\RHS}-\flexhat{\RHS}=\flexhat{a}+1-\frac{(\flexhat{a}+1)^2}{\flexhat{V}(\flexhat{V}+\flexhat{a})},
\end{equation*}
which coincides  with the corresponding difference ${\LHS}-\flexhat{\LHS}$ in Equation~\eqref{eq:diffLHS2}.

\begin{figure}[b]
\centering
\includegraphics[]{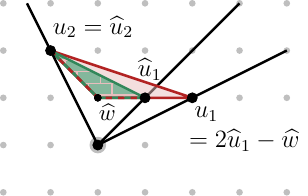}
\caption{{\bf Case II -- Illustration of considered domains in decomposition of 
$\boldsymbol{\nab \setminus} \text{sail}_{\boldsymbol{\sigma}}$.}
Here: $V=5, a=3$ and $\protect\flexhat{V}=3, \protect\flexhat{a}=1$. 
$\nabhat \setminus \sailhat$ (green) and $\nab \setminus \sail$ (red).}
\label{fig:rhscaseiisail}
\end{figure}

\bigskip
\paragraph{\bf Case II}
By our assumption (as illustrated in Figure~\ref{fig:rhscaseiisail}), we have
\begin{align*}
\nab\setminus\sail= (\nabhat\setminus\sailhat)\cup\conv(\raygenhat_1,\raygenhat_2,2\raygenhat_1-\flexhat{w})\,.
\end{align*}
Using again $\det(\raygenhat_1,\raygenhat_2)=\flexhat{V}$, we obtain
\begin{align*}
&v\left({\nab}\setminus\sail\right)-v\left(\nabhat\setminus\sailhat\right) 
=v\left(\conv(\raygenhat_1,\raygenhat_2,2\raygenhat_1-\flexhat{w})\right)
=\det(\raygenhat_1-\flexhat{w},\raygenhat_2-\raygenhat_1)\\
&=\det\left(\left(1-\frac{\flexhat{a}}{\flexhat{V}}\right)\raygenhat_1-\frac{1}{\flexhat{V}}\raygenhat_2,\raygenhat_2-\raygenhat_1\right)
=\left(1-\frac{\flexhat{a}}{\flexhat{V}}\right)\cdot \flexhat{V}-\frac{1}{\flexhat{V}}\cdot \flexhat{V}
=\flexhat{V}-\flexhat{a}-1\,.
\end{align*}

As in Case~I, we need to determine how the functionals involved in the second summand of the RHS of Equation~\eqref{eq:sail} behave when we change from $\flexhat{\cone}$ to ${\cone}$.
First recall that $\mstarhat=-\raygenhat_1^*-\raygenhat_2^*$. As ${\raygen}_1=2\raygenhat_1-\flexhat{w}$ and ${\raygen}_2=\raygenhat_2$, we have
\begin{align*}
\mstar=-\frac{\flexhat{V}+1}{2\flexhat{V}-\flexhat{a}}\raygenhat_1^*-\raygenhat_2^*
=\mstarhat+\frac{\flexhat{V}-\flexhat{a}-1}{2\flexhat{V}-\flexhat{a}}\raygenhat_1^*\,.
\end{align*}

\begin{figure}[b]
\centering
\includegraphics{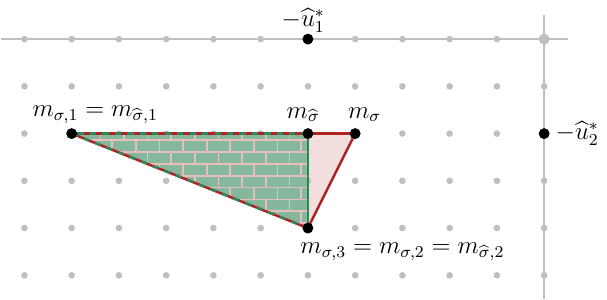}
 \caption{{\bf Case II -- dual perspective} for illustration in Figure~\ref{fig:rhscaseiisail} with 
$\cup_{i=1}^{\ksighat-1} \convmhat{i}$ (green) and  
$ \cup_{i=1}^{\ksig-1}\convm{i}$ (red).}
\label{fig:rhscaseiidual}
\end{figure}

As above, we enumerate the edges of the sail that are not incident to $0$ starting from the edge incident to $\raygenhat_2$ (finishing with the edge incident to $\raygenhat_1$).
By construction, $\sailclosed$ has the same edges as $\overline{\text{sail}}_{\flexhat{\cone}}$ and an additional edge with functional $\mstark{\ksighat+1}=\mstarhatk{\ksighat}$. The other functionals are identical, see Figure~\ref{fig:rhscaseiidual} for an illustration.
As in Case I, the functionals $\mstarhatk{1}$ and $\mstarhatk{\ksighat}$ can be expressed by Equation~\eqref{eq:m1} and \eqref{eq:mk}.
Hence,
\begin{align*}
&\sum\limits_{i=1}^{\ksig-1} \nvol(\convm{i}) -\sum\limits_{i=1}^{\ksighat -1} \nvol(\convmhat{i})\\
&=\sum\limits_{i=1}^{\ksighat-1} \nvol(\conv(\mstar,\mstarhatk{i},\mstarhatk{i+1})) -\sum\limits_{i=1}^{\ksighat-1} \nvol(\conv(\mstarhat,\mstarhatk{i},\mstarhatk{i+1}))\\
&= \nvol(\conv(\mstar,\mstarhat,\mstarhatk{\ksighat}))
=\det(\mstar-\mstarhat,\mstarhatk{\ksighat}-\mstarhat)\\
&=\det\left(\frac{\flexhat{V}-\flexhat{a}-1}{2\flexhat{V}-\flexhat{a}}\raygenhat_1^*,(\flexhat{V}-\flexhat{a}-1)\raygenhat_2^*\right)
=\frac{(\flexhat{V}-\flexhat{a}-1)^2}{2\flexhat{V}-\flexhat{a}}\det(\raygenhat_1^*,\raygenhat_2^*)
=\frac{(\flexhat{V}-\flexhat{a}-1)^2}{\flexhat{V}(2\flexhat{V}-\flexhat{a})}\,.
\end{align*}
Combining everything yields
\begin{align*} 
{\RHS}-\flexhat{\RHS}=\flexhat{V}-\flexhat{a}-1+\frac{(\flexhat{V}-\flexhat{a}-1)^2}{\flexhat{V}(2\flexhat{V}-\flexhat{a})}\,,
\end{align*}
which coincides  with the corresponding difference ${\LHS}-\flexhat{\LHS}$ in Equation~\eqref{eq:diffLHS1}.

\section*{Acknowledgements}
The first author is partially supported by the Deutsche Forschungsgemeinschaft (DFG -- German Research Foundation) -- Project-ID 195170736 -- TRR109 “Discretization in Geometry and Dynamics”.
The second and third authors have been partially
supported by the Polish-German grant 
"ATAG -- Algebraic Torus Actions: Geometry and Combinatorics" 
[project number 380241778] of the Deutsche
Forschungsgemeinschaft (DFG). 

\bibliographystyle{amsalpha}

\end{document}